\documentclass[11pt]{article}
\usepackage{amsmath,amssymb,amsthm,amsfonts,mathrsfs}
\usepackage{hyperref}

\theoremstyle{definition}

\theoremstyle{plain}
\newtheorem{thm}{Theorem}[section]
\newtheorem{lem}[thm]{Lemma}

 \newtheorem{preconj}{{\bf  Theorem}}

\newenvironment{conj}{\begin{preconj}{\hspace{-0.5
               em}{\bf }}}{\end{preconj}}

\title{\large \bf Nilpotent and polycyclic-by-finite maximal subgroups of skew linear groups}
\author{\small M. Ramezan-Nassab ${}^{a}$, D. Kiani  ${}^{a,}$\thanks{Corresponding author\newline
${}^{a}$ Department of Mathematics and Computer Science,
Amirkabir University of Technology (Tehran Polytechnic), P.O. Box:
15875-4413, Tehran, Iran\newline E-mail Addresses:
ramezann@aut.ac.ir, dkiani@aut.ac.ir}}

\date{}

\begin{document}
\maketitle
\begin{abstract}
Let $D$ be an infinite division ring, $n$ a natural number and $N$ a subnormal subgroup
of $\mathrm{GL}_n(D)$ such that  $n=1$ or the center of $D$ contains at
least five elements. This paper contains two main results. In the first one we prove that
each nilpotent maximal subgroup of $N$ is abelian; this generalizes the result in
[R. Ebrahimian, J. Algebra 280 (2004) 244--248] (which asserts that each maximal
subgroup of $\mathrm{GL}_n(D)$ is abelian) and a result
in [M. Ramezan-Nassab, D. Kiani, J. Algebra 376 (2013) 1--9].
In the second one we show that a maximal subgroup of $\mathrm{GL}_n(D)$ cannot be
polycyclic-by-finite.
\end{abstract}
\section{Introduction}
Throughout this paper $D$ denotes a division ring, $n$ is a natural
number, $M_n(D)$ is the full $n\times n$ matrix ring over $D$
and $\mathrm{GL}_n(D)$ is the group of units of $M_n(D)$.
The maximal soluble, maximal nilpotent, and maximal locally nilpotent subgroups
of general linear groups (over fields) were extensively
studied by Suprunenko; the main results are expounded in \cite{sup}.

Our object here is to discuss the general skew linear groups
whose maximal subgroups are of some special types.
Some properties of maximal subgroups of $\mathrm{GL}_n(D)$ have been
studied in a series of papers, see, e.g.,
\cite{akbari n, sol, nil, kiani, j.algebra, engel}.
In all of those papers, authors attempted to show that the structure of maximal subgroups of
$\mathrm{GL}_n(D)$ is similar, in some sense, to the structure of
$\mathrm{GL}_n(D)$. For instance, if $D$ is an infinite division
ring, in \cite{nil} it was shown that every nilpotent
maximal subgroup of $\mathrm{GL}_n(D)$ is abelian, and in \cite{engel}
the authors proved that for $n\geq 2$, every locally nilpotent
maximal subgroup of $\mathrm{GL}_n(D)$ is abelian. Also, if
$D$ is non-commutative and $n\geq 2$, in \cite{sol} it  was shown that every soluble
maximal subgroup of $\mathrm{GL}_n(D)$ is abelian, and in \cite{engel}
the authors proved that for $n\geq 3$, every locally soluble
maximal subgroup of $\mathrm{GL}_n(D)$ is abelian. For some recent results see~\cite{j.algebra}.

This paper contains two main results. In section~2, instead of maximal subgroups of
$\mathrm{GL}_n(D)$, we consider maximal subgroups of subnormal subgroups of $\mathrm{GL}_n(D)$.
The structure of such groups have been investigated in various papers, see, e.g.,
\cite{polynomial, fc subnormal,j.algebra, nilpotent}. As mentioned earlier, if $D$ is an infinite division
ring, in \cite{nil} it was shown that every nilpotent
maximal subgroup of $\mathrm{GL}_n(D)$ is abelian. In~\cite[Corollary~1]{fc subnormal}, the authors
proved that if $D$ is a finite-dimensional division algebra over its center, then
every nilpotent maximal subgroup of a subnormal subgroup of $\mathrm{GL}_n(D)$ is abelian.
In~\cite[Theorem~4]{j.algebra}, the author showed that (without any condition on dimension)
every nilpotent maximal subgroup of a subnormal subgroup of $\mathrm{GL}_n(D)$ is metabelian.
Here, we generalize those results and show that every nilpotent maximal subgroup of a
subnormal subgroup of $\mathrm{GL}_n(D)$ is abelian. More precisely, in section~2 we prove the
following theorem.

\begin{conj}
Let $D$ be an infinite division ring, $n$ a natural number, $N$ a subnormal subgroup
of $\mathrm{GL}_n(D)$ and $M$ a nilpotent maximal subgroup of $N$.
If $n=1$ or the center of $D$ contains at least five elements, then $M$ is abelian.
\end{conj}

In section~3, we consider polycyclic-by-finite skew linear groups.
For the reasons for doing this, see \cite[Chapter~4]{skew} and references there. We will see that $\mathrm{GL}_n(D)$
cannot be polycyclic-by-finite (Lemma~\ref{lem3}). In this direction, we show that maximal
subgroups of $\mathrm{GL}_n(D)$ have the same property, i.e., maximal subgroups of  $\mathrm{GL}_n(D)$
are not  polycyclic-by-finite. In fact we have:

\begin{conj}
Let $D$ be an infinite division ring, $n$ a natural number and $M$ a maximal subgroup of $\mathrm{GL}_n(D)$.
If $n=1$ or the center of $D$ contains at least five elements, then $M$ cannot be polycyclic-by-finite.
\end{conj}

Note that all division rings which are of characteristic zero, or algebraic over their centers
have at least five elements in their centers. However, it seems that in Theorems~A and B (for $n\geq 2$)
the condition that the center of $D$ contains at least five elements is not necessary (see also
the remarks after Theorem~\ref{poly2}).

Our notation is standard. To be more precise, $F$ always denotes the center of the
division ring $D$ unless stated otherwise.
We shall identify the center $FI_n$ of $M_n(D)$ with
$F$. If $D$ has at least four elements, for $n\geq 2$
we denote by $\mathrm{SL}_n(D)$ the derived subgroup of $\mathrm{GL}_n(D)$.
Let $G$ be a subgroup of $\mathrm{GL}_n(D)$. We denote by
$F[G]$ the $F$-linear hull of $G$, i.e., the $F$-algebra
generated in $M_n(D)$ by elements of $G$ over $F$. If $n=1$, then
$F(G)$ is the division ring generated in $D$ by $F$ and $G$; note
that if each element of $G$ is algebraic over $F$, then
$F(G)=F[G]$. If $D^n$ is the space of row $n$-vectors over $D$,
then $D^n$ is a $D$-$G$ bimodule in the obvious manner. We say
that $G$ is irreducible, reducible, or completely reducible,
whenever $D^n$ has the corresponding property as $D$-$G$
bimodule. Also, $G$ is called absolutely irreducible if $F[G]
=M_n(D)$. The derived subgroup of $G$ denotes by $G'$.
For a given ring $R$, the group of units of $R$ is denoted
by $R^*$. Let $S$ be a subset of $R$, then the centralizer of $S$ in $R$ is
denoted by $C_R(S)$.
\section{Nilpotent maximal subgroups}
In this section we prove Theorem~A. First we assert some useful lemmas which are also
used in the next section.

\begin{lem}\label{gold1} {\rm(\cite[Corollary~24]{wehr89})}
Let $A$ be a one-sided Artinian ring. Suppose $S$ is a
right Goldie subring of $A$ and $G$ is a locally soluble subgroup of the group
of units of $A$ normalizing $S$. Set $R = S[G]\leq A$ and assume $R$ is prime.
Then $R$ too is right Goldie.
\end{lem}
\begin{lem}\label{metabelian}
Let $D$ be an infinite-dimensional division algebra over its center, $N$ a subnormal subgroup
of $D^*$, and $M$ a maximal subgroup of $N$. If $M$ is metabelian, then it is abelian.
\end{lem}
\begin{proof}
Since $M'$ is abelian, we can find a maximal normal abelian subgroup $A$ of $M$ containing $M'$.
Suppose on the contrary that $A\neq M$. If $T$ is a subgroup of $M$ such that $A\lneqq T$, we claim that $F(T)=D$. In fact,
we have $M\subseteq  N_N(F(T)^*)\subseteq N$. If $M= N_N(F(T)^*)$, then $F(T)^*\cap N\leq M$,
so $F(T)^*\cap N$ as a metabelian subnormal subgroup of $F(T)^*$ is abelian, so $T$ is also abelian,
this contradicts the choose of $T$. Therefore, by the maximality of $M$ in $N$ we may assume
$N_N(F(T)^*)=N$. Then $N\subseteq N_{D^*}(F(T)^*)$, which
by~\cite[13.3.8]{scott} we have $F(T)=D$, as claimed.

Setting $K=F(A)$, clearly $M\nsubseteq K$.  Suppose there is some $a\in M\setminus K$ which is transcendental
over $K$ and set $T=A\langle a^2\rangle$.
By~Lemma~\ref{gold1}, $F[T]$ is a Goldie ring; since it is also a domain,
it is an Ore domain. On the other hand,
by the fact that $a^2\notin A$ we conclude that $T$ is a subgroup of $M$
properly containing $A$; hence by what we proved before we conclude
that $F(T)=D$. Hence the division ring generated by
$F[T]$, which is exactly its classical ring of quotients, coincides with $D$.
Thus there exist two elements $s_1,s_2\in F[T]$ such that $a=s_1s_2^{-1}$.
Write $s_1=\sum_{i=l}^m k_ia^{2i}$ and $s_2=\sum_{i=l}^m k_i'a^{2i}$,
where $k_i, k_i'\in K$, for any $l\leq i\leq m$. Hence
$$\sum_{i=l}^m ak_i'a^{2i}=\sum_{i=l}^m k_ia^{2i}.$$
If we set $l_i=ak_i'a^{-1}$, for any $1\leq i\leq m$, then $l_i$'s are elements of $K$
and we have
$$\sum_{i=l}^m l_ia^{2i+1}=\sum_{i=l}^m k_ia^{2i},$$
which shows that $a$ is algebraic over $K$, a contradiction.

Now let $x\in M\setminus K$ be algebraic
over $K$. Assume that $x$ satisfies an equation of the form
$\sum_{i=0}^n k_i x^i=0$, where $k_i\in K$ for any $0\leq i\leq n$ and $k_n=1$. Using the fact that $x$
normalizes $K$ and the above equality one can easily show that $R=\sum_{i=0}^n Kx^i$ is a ring
that is of finite-dimension as a left vector space over $K$. Therefore it is a division ring. If
we set $T=A\langle x\rangle$, by what we proved before $F(T)=D$. On the other
hand obviously we have $F(T)=R$. Therefore $[D:K]_l<\infty$. Thus
$D$ is a finite-dimensional division algebra over its center. This contradiction shows that
$M$ is abelian.
\end{proof}

In \cite{akbari n} it was proved that $\mathbb{C^*\cup C^*}j$
is a maximal subgroup of the real quaternion division algebra. Clearly
$(\mathbb{C^*\cup C^*}j)'\subseteq \mathbb{C^*}$ and so $\mathbb{C^*\cup C^*}j$
is metabelian but not abelian. Thus in~Lemma~\ref{metabelian}, the condition that $D$ is
of infinite-dimensional cannot be removed.

Now, using Lemma~\ref{metabelian}, we can prove Theorem~A for $n=1$.

\begin{thm}\label{thm1}
Let $D$ be a division ring and $N$ a subnormal subgroup
of $D^*$. Then every nilpotent maximal subgroup of $N$ is abelian.
\end{thm}
\begin{proof}
Let $M$ be a nilpotent maximal subgroup of $N$. By~\cite[Proposition~1.1]{nilpotent}
$M'$ is an abelian group. Thus by~Lemma~\ref{metabelian}, we may assume that
$D$ is a finite-dimensional division ring, which by~\cite[Corollary~1]{fc subnormal}
we conclude the result.
\end{proof}

The proof of Theorem~A, for $n\geq 2$, needs some different approaches. To this end,
we have following two lemmas.
\begin{lem}\label{main}
Let $D$ be a non-commutative division ring with center $F$, $N$ a subnormal subgroup
of $\mathrm{GL}_n(D)$, and $M$ an absolutely irreducible and maximal subgroup of $N$.
If $M/(M\cap F^*)$ is locally finite, then for any normal subgroup $H$ of $M$ we
either have $H\subseteq F^*$, $F[H]=M_n(D)$, or $F[H]\simeq F_1\times\cdots\times F_s$
for some natural number $s$ and fields $F_i\supseteq F$.
\end{lem}
\begin{proof}
Since $M\subseteq N_N(F[H]^*)\subseteq N$,
the maximality of $M$ in $N$ implies that, either $N_N(F[H]^*)=N$ or $N_N(F[H]^*)=M$.
Let $G:=F[H]^*\cap N$. In the first case, $G$ is a subnormal subgroup of $\mathrm{GL}_n(D)$.
Consequently, if $H$ is not central, $\mathrm{SL}_n(D)\subseteq G\subseteq F[H]^*$
by~\cite[Lemma~2.3]{gon}, and thus the Cartan-Brauer-Hua theorem for matrix ring implies that $F[H]=M_n(D)$.

Next assume $N_N(F[H]^*)=M$. Then $G\leq M$ and $G$ is a normal subgroup of
$F[H]^*$. On the other hand, since $F[M]=M_n(D)$ and $M/(M\cap F^*)$ is locally finite,
$D$ is a locally finite-dimensional division algebra over $F$. Also, by Clifford's theorem
$H$ is complectly reducible; therefore $F[H]$ is semisimple Artinian  by~\cite[p.~7]{skew}.
Thus, by the Wedderburn-Artin theorem, there exist natural numbers $n_i$ and division rings $D_i$ such that
$$F[H]\simeq M_{n_1}(D_1)\times\cdots\times M_{n_s}(D_s),$$
as $F$-algebras. Now,
$$G\unlhd F[H]^*\simeq \mathrm{GL}_{n_1}(D_1)\times\cdots\times \mathrm{GL}_{n_s}(D_s).$$
Define $\pi_i:G\longrightarrow \mathrm{GL}_{n_i}(D_i)$ by
$\pi_i((g_1,\ldots,g_s))=g_i$, for $1\leq i\leq s$. Clearly,
$\pi_i$ is a group homomorphism. Let $G_i:=\pi_i(G)$. Thus, for
every $i$, $G_i$ is a normal subgroup of $\mathrm{GL}_{n_i}(D_i)$ which is
locally finite over $F_i^*:=Z(D_i)^*\supseteq F^*$ .
If there exists some $i$ such that $G_i$ is non-abelian, then $n_i\geq 2$: for if
$n_i=1$, as $G_i/Z(G_i)$ is locally finite, $G_i'$ is a torsion subnormal subgroup
of $D_i^*$, which by~\cite[Theorem~8]{her} $G_i'\subseteq F_i^*$, which implies that $G_i$
is soluble, and so $G_i$ is also central by~\cite[14.4.4]{scott}, a contradiction. Then, by~\cite[p.~154]{skew},
$D_i$ is a locally finite field and hence $F$ is also locally finite.
Thus $D$ is algebraic over a finite field and hence by Jacobson's
theorem in~\cite[p. 208]{lam}, we obtain $D=F$, which is a
contradiction. Therefore, for every $i$, $G_i$ is abelian and so
$G$ (thus $H$) is an abelian group. This implies that $F[H]\simeq F_1\times\cdots\times F_s$,
and completes the proof.
\end{proof}
\begin{lem} \label{lanski} {\rm(\cite[Theorem~2]{lanski})}
Let $R$ be a prime ring with $1$, $Z=Z(R)$ be the
center of $R$ containing at least five elements,
and $\overline{U}$ the $Z$-subalgebra of $R$ generated by
$R^*$. Assume that $\overline{U}$ contains a
nonzero ideal of $R$. If $N$ is a soluble normal subgroup of $R^*$, then either $R$ is a domain
or $N\subseteq Z$.
\end{lem}
We are now in a position to complete the proof of Theorem~A as follows.
\begin{thm}\label{nil2}
Let $D$ be an infinite division ring, $N$ a subnormal subgroup
of $\mathrm{GL}_n(D)$, $n\geq 2$, and $M$ a nilpotent maximal subgroup of $N$.
If the center of $D$ contains at least five elements, then $M$ is abelian.
\end{thm}
\begin{proof}
By~\cite[Corollary~1]{fc subnormal} we may assume $D$ is infinite-dimensional
over $F$. Let $R:=F[M]$. Since $M\subseteq R\cap N\subseteq N$, by the maximality of $M$ in $N$ we consider
the following two cases:\vspace{.3cm}\\
{\bf Case~1.} Suppose $M=R\cap N$. Then $M$ is a normal subgroup of $R^*$. On the other hand,
if $M$ is reducible, then it contains an isomorphic copy of $D^*$ by~\cite[Lemma~1]{polynomial};
so $D$ is a field, a contradiction. Assume
that $M$ is irreducible; thus $R$ is a prime ring by~\cite[1.1.14]{skew}
and  Goldie by~Lemma~\ref{gold1}.
Moreover, since the $Z(R)$-subalgebra of $R$ generated
by $R^*$ is $R$ itself, we can use of~Lemma~\ref{lanski}
to deduce that either $R$ is a domain or $M\subseteq Z(R)$. But, in the first case
$R$ is in fact an Ore domain.
Denote the classical quotient ring of $R$ by $\Delta$; then $\Delta$ is a
division ring contained in $M_n(D)$ by~\cite[Theorem~5.7.8]{skew}.
If $N=\Delta\cap N$, then $N\subseteq\Delta^*$,
so the Cartan-Brauer-Hua theorem for matrix ring implies that $\Delta=M_n(D)$ which is
impossible since $n\geq 2$. Therefore $M=\Delta\cap N$; thus $M$ as a nilpotent normal
subgroup of $\Delta^*$ is abelian.\vspace{.3cm}\\
{\bf Case~2.} In this case, we consider the case $N=R \cap N$.
Thus $\mathrm{SL}_n(D)\subseteq N\subseteq R^*$, so $R=M_n(D)$.
Therefore, $M$ is center-by-(locally finite) by~\cite[Theorem~5.7.11]{skew}. Clearly
$Z(M)=M\cap F^*$, so $M/(M\cap F^*)$ is locally finite.

First we assume $M'\subseteq F^*$. Given $x, y\in M$ such that $xy\neq yx$, we
have $F^*\langle x, y\rangle \unlhd M$ (note that we may assume $F^*\subseteq M$,
since otherwise, we can replace $M$ by $F^*M$ and $N$ by $F^*N$). Hence,
by~Lemma~\ref{main}, $F[\langle x,y\rangle]=M_n(D)$. Since $x$ and $y$ are algebraic over $F$,
we have $F[\langle x, y\rangle]=F[x,y]$, consequently, $F[x,y]=M_n(D)$. So $[D:F]<\infty$ since $D$
is locally finite-dimensional over $F$; this contradicts our assumption and proves
that $M$ is abelian.

Next we assume $M'\not\subseteq F^*$. Let $K=F[M']$. Then by~Lemma~\ref{main} and
\cite[Lemma~11]{j.algebra},  $K=F_1\times\cdots\times F_s$
for some natural number $s$ and fields $F_i\supseteq F$. Suppose $x\in M'\setminus F^*$ and
let $f(t)\in F[t]$ be the minimal polynomial of $x$ over $F$.
Since $M\subseteq N_{\mathrm{GL}_n(D)}(K)$, every conjugate of $x$ with
respect to $M$ is in $K$ and as well is a root of $f(t)$. Since $f(t)$ has a finite number of roots in $K$ we
have $[M:C_M (x)] <\infty$. Therefore, there is a normal subgroup
$H$ of $M$ such that $H\subseteq C_M(x)$, and
$|M/H|<\infty$. If $F[H]=M_n(D)$, since $H\subseteq C_M (x)$, every element of
$M_n(D)$ commutes with $x$. So $x\in F$ which conflicts with the choice of
$x$. Thus by~Lemma~\ref{main} we may assume $H$ is abelian.
Therefore, $M$ is abelian-by-finite and so, by
Lemma~1.11 of \cite[p. 176]{passman}, the group ring $FM$ satisfies
a polynomial identity. Therefore $F[M]=M_n(D)$ as a homomorphic image
of $FM$, satisfies a polynomial identity too. So by Kaplansky's theorem
in~\cite[p. 36]{rowen} we conclude that $[D:F]<\infty$, a contradiction.
This finishes the proof.
\end{proof}
\section{Polycyclic-by-finite maximal subgroups}
The principal aim of this section is to prove Theorem~B. The main step in the proof
is to show that the result holds in the case $n=1$. For our purposes, we need several
lemmas as follows.
\begin{lem}\label{lem3}
Let $D$ be an infinite division ring. Then $\mathrm{GL}_n(D)$ cannot be a polycyclic-by-finite group.
\end{lem}
\begin{proof}
Suppose $H$ is a polycyclic normal subgroup of $\mathrm{GL}_n(D)$
and $\mathrm{GL}_n(D)/H$ is finite. It is known that $H$ must be central, so
$\mathrm{GL}_n(D)/F^*$ is finite and therefore $D^*/F^*$ is a torsion group. Consequently
$D=F$ is an infinite field. On the other hand, since every polycyclic-by-finite group
is finitely generated, $\mathrm{GL}_n(D)$ is finitely generated. This is impossible:
for $n\geq 2$ use \cite[Corollary~1]{f.g}, and for $n=1$ use the fact that
the multiplicative group of a field cannot be finitely generated unless the field is finite.
\end{proof}
\begin{lem}\label{a-by-f}
Let $D$ be an infinite division ring and $M$ be a
polycyclic-by-finite maximal subgroup of $\mathrm{GL}_n(D)$. Then $M$ cannot be abelian-by-finite.
\end{lem}
\begin{proof}
By~\cite[Corollary~3]{f.g}, we may assume $D$ is of infinite-dimensional division algebra
over its center. Suppose on the contrary that $M$ is abelian-by-finite.
Then, by Lemma~1.11 in~\cite[p. 176]{passman}, $F[M]$
satisfies a polynomial identity. If $F[M]=M_n(D)$, we use Kaplansky's theorem
in~\cite[p. 36]{rowen} to obtain $[D:F]<\infty$, a contradiction.

Now suppose $F[M]^*=M$. By~\cite[Lemma~1]{polynomial}, $M$ is irreducible
and so $F[M]$ is a prime ring. Let $F_1:=C_{M_n(D)}(M)$ and recall that $F_1$ is a division ring by~\cite[Lemma~8]{akbari n}.
We claim that $F_1$ is a field. Let $x\in F_1'$. Now, by the maximality of $M$ in $\mathrm{GL}_n(D)$,
either $\langle x,M \rangle =M$ or $\langle x,M \rangle = \mathrm{GL}_n(D)$.
In the first case we have $x\in M\cap F_1$ and so $x\in Z(M)$. In
the second case we obtain $x\in F^*$. Hence in any case we have
$x\in F^*Z(M)$ and so $F_1'\subseteq F^*Z(M)$. This means that
$F_1'$ is abelian. So, $F_1^*$ is soluble and hence $F_1$ is a field and the claim is established.
Also, by the maximality of $M$ in $\mathrm{GL}_n(D)$ (and by similar argument as in
the proof of~Lemma~\ref{lem3}), we may assume $F_1^*\subseteq M$. Consequently, $F[M]$ is a prime PI-ring
whose center $F_1$ is a field and therefore, by~\cite[Corollary~1.6.28]{rowen}, it is a simple
ring. So, again by Kaplansky's theorem we
have $F_1[M]\simeq M_m(\Delta)$ for some natural number $m$ and a division
ring $\Delta$. Thus $M=F_1[M]^*\simeq\mathrm{GL}_m(\Delta)$. If $\Delta$ is finite,
$M$ is also finite; this is impossible by~\cite[Lemma~9]{akbari n}. Thus $\Delta$ is an
infinite division ring; but this contradicts Lemma~\ref{lem3} and completes the proof.
\end{proof}
\begin{lem}\label{wehr08} {\rm(\cite[3.11]{wehr08})}
Let $G$ be a locally nilpotent subgroup of the multiplicative group
$D^*$ of the division ring $D$. Suppose also that $H=N_{D^*}(G)$,
$E=C_D(G)$, and $D=E(G)$. Denote the maximal $2$-subgroup of $G$ by $Q$.
Then one of the following holds:
\begin{enumerate}
  \item [\rm(i)] $T$ (the maximal locally finite normal subgroup of $G$) is abelian and $H/GE^*$ is abelian;
  \vspace{-.18cm}
  \item [\rm(ii)] $G=Q\cdot C_G(Q)$ where $Q$ is quaternion of order $8$ and
$H/GE^*\simeq Sym(3)\times Y$ for $Y$ abelian;
\vspace{-.18cm}
  \item [\rm(iii)] $G\neq Q\cdot C_G(Q)$ where $Q$ is quaternion of order $8$ and $H/GE^*$
is abelian;
\vspace{-.18cm}
  \item [\rm(iv)] $Q$ is non-abelian with $|Q|>8$ and $H/GE^*$ has an abelian subgroup
$Y$ with index in $H/GE^*$ at most $2$ ($1$ if $Q$ is infinite).
\end{enumerate}
\end{lem}
\begin{lem}\label{p.f}
Let $D$ be an infinite division ring, $M$ a polycyclic-by-finite maximal subgroup
of $D^*$ and $G$ a nilpotent normal subgroup of $M$. Then $G$ is abelian and $F(G)^*\unlhd M$.
\end{lem}
\begin{proof}
Since $M\subseteq N_{D^*}(F(G)^*)$, we either have $F(G)^*\unlhd M$ or
(by the Cartan-Brauer-Hua theorem) $F(G)=D$. In the former case, as
the multiplicative group of the division ring $F(G)$ is polycyclic-by-finite,
$G$ is abelian.

We claim that $F(G)\neq D$. Assume on the contrary $F(G)=D$.
If $M$ is absolutely irreducible, then $M/C_M(G)$ is torsion by~\cite[Theorem~5.7.11]{skew}.
Since $C_M(G)\subseteq F^*$ (because of $F(G)=D$), we conclude that
$M$ is torsion over $F$ and therefore $F[G]=F(G)=D$, i.e., $G$ is absolutely irreducible.
Clearly, $Z(G)=G\cap F$; so $G/(G\cap F)$ is locally finite by~\cite[Theorem~5.7.11]{skew}.
This implies that $D$ is locally finite-dimensional over $F$.
Since $M$ is finitely generated, we may assume $M=\langle m_1,\ldots,m_s\rangle$.
So, $F[m_1,\ldots,m_s]=F[\langle m_1,\ldots,m_s\rangle]=D$ implies that $[D:F]<\infty$. This
conflicts \cite[Corollary~3]{f.g}.

Now, suppose $M$ is not absolutely irreducible,
so $F[M]^*=M$. Since $F(G)=D$, $F=C_D(G)$. On the other hand, $M\subseteq N_{D^*}(G)\subseteq D^*$.
If $N_{D^*}(G)=D^*$, then $G$ as a nilpotent normal subgroup of
$D^*$ is central, so $D=F$ which is impossible. So assume $M=N_{D^*}(G)$. Now we can
apply Lemma~\ref{wehr08}. Denote the maximal 2-subgroup of $G$ by $Q$.
If $Q$ is finite, then $F[Q]^*\subseteq F[M]^*=M$
implies that the multiplicative group of the division ring
$F[Q]$ is polycyclic-by-finite which asserts that $Q$ is abelian. Thus by~Lemma~\ref{wehr08}
we may assume that $M/GF^*$ is abelian. This gives us $M'\subseteq GF^*$ is nilpotent.
Therefore $M$ is soluble, so it is abelian by~\cite[Theorem 3.7]{sol}; this cannot happens
by~Lemma~\ref{a-by-f}. Therefore our claim, and so the statement of the lemma, holds.
\end{proof}
\begin{lem}\label{wehr07} {\rm(\cite[Proposition~4.1]{wehr07})}
Let $D=E(M)$ be a division ring generated as
such by its metabelian subgroup $M$ and its division subring $E$ such that $E\subseteq C_D(M)$.
Set $K=N_{D^*}(M)$, $G=C_M(M')$, $T$ the maximal periodic normal subgroup of $G$, $F=E(Z(G))$, $L=N_{F^*}(M)=K\cap F$.
Then
\begin{enumerate}
  \item [\rm(i)] if $M$ has a quaternion subgroup $Q$ of order $8$ with
$M=QC_M(Q)$, then $K=Q^+ML$;
  \vspace{-.18cm}
  \item [\rm(ii)] if $T$ is abelian and contains an element $x$ of order
$4$ not in the center of $G$, then $K=\langle 1+x\rangle ML$;
\vspace{-.18cm}
  \item [\rm(iii)]  in all other cases $K=ML$.
\end{enumerate}
\end{lem}
We are now ready to prove Theorem~B in the case $n=1$.
\begin{thm}\label{thm3}
Let $D$ be an infinite division ring and $M$ a maximal subgroup
of $D^*$. Then $M$ cannot be polycyclic-by-finite.
\end{thm}
\begin{proof}
Let $M$ be a polycyclic-by-finite group.  We have a series of the form
$$1=H^{(s)}\lhd\cdots\lhd H'\lhd H\lhd M,$$
where $M/H$ is a finite group. By~Lemma~\ref{a-by-f}, $H$ is non-abelian.
Set $H^{(0)}=H$, and let $r$ be the largest integer such that
$H^{(r)}\nsubseteq F$, and so $H^{(r+1)}\subseteq F$. Note that $H^{(r)}$ is a nilpotent
normal subgroup of $M$, so $H^{(r)}$ is abelian by~Lemma~\ref{p.f}. Let $M_1:=H^{(r-1)}$
which is a (non-abelian) metabelian normal subgroup of $M$. Since $M\subseteq N_{D^*}(F(M_1))$,
we have $F(M_1)=D$. If we set $G=C_{M_1}(M_1')$, then clearly $G$ is a nilpotent normal
subgroup of $M$; thus by~Lemma~\ref{p.f}, $G$ is abelian and $F(G)^*\unlhd M$.
Now, by~Lemma~\ref{wehr07}, we have the following three cases to consider:
\begin{enumerate}
  \item [\rm(i)] $M_1=QC_{M_1}(Q)$. Then $C_{M_1}(Q)\lhd M_1$ and hence we conclude
that $M_1/C_{M_1}(Q)\simeq Q/(Q\cap C_{M_1}(Q))=Q/Z(Q)$
is abelian. Thus, $M_1'\subseteq C_{M_1}(Q)$ and so
$Q\subseteq C_{M_1}(M_1')=G$, which is a contradiction since $G$ is abelian.
  \vspace{-.18cm}
  \item [\rm(ii)] The case (ii) of~Lemma~\ref{wehr07} cannot occur since $G$ is abelian.
\vspace{-.18cm}
  \item [\rm(iii)] $M=M_1F(G)^*$. In this case $M/F(G)^*\simeq M_1/(F(G)^*\cap M_1)$ is abelian because
$M_1'\subseteq F(G)^*\cap M_1$, and hence $M'\subseteq F(G)^*$ is abelian; consequently
$M$ is abelian by~Lemma~\ref{metabelian}. Since $M_1$ was non-abelian, we arrive at a contradiction.
\end{enumerate}
The proof of the theorem is completed.
\end{proof}
Finally we assert a theorem which completes the proof of Theorem~B.
\begin{thm}\label{poly2}
Let $D$ be an infinite division ring and $M$ a maximal subgroup
of $\mathrm{GL}_n(D)$, $n\geq 2$. If the center of $D$ contains at least five elements,
then $M$ cannot be polycyclic-by-finite.
\end{thm}
\begin{proof}
Let $M$ be a polycyclic-by-finite group. If $M$ is absolutely irreducible,
$M$ is abelian-by-finite by~\cite[Theorem~1~(i)]{j.algebra}. This cannot happens by~Lemma~\ref{a-by-f}.
Now let $F[M]^*=M$. Since $M$ is polycyclic-by-finite, $F[M]$ is a Noetherian ring by
a result of Hall (see, e.g., \cite[p.~35]{wehr2009}). On the other hand, by~Lemmas~\ref{lanski} and \ref{a-by-f},
$F[M]$ is a domain. Then, $F[M]$ is in fact an Ore domain. Therefore, the
classical quotient ring of $F[M]$ is a division ring $\Delta$ which by~\cite[Theorem~5.7.8]{skew},
$\Delta$ is contained in $M_n(D)$. Since $n\geq 2$, maximality of $M$ in $\mathrm{GL}_n(D)$
implies that $M=F[M]^*=\Delta^*$, and so $M$ is abelian. This, again, contradicts Lemma~\ref{a-by-f}.
\end{proof}
We close this paper by some remarks. Firstly, from our proofs of Theorems~\ref{nil2} and \ref{poly2},
we see that in those theorems it is enough that (instead of the center of $D$) the center of
$F[M]$ contains at least five elements. Secondly, by the work in section~3 we see that the statement of
Theorem~B (instead of polycyclic-by-finite groups) holds also for finitely generated
soluble-by-finite groups:\vspace{.3 cm}\\
{\bf Theorem~B$'$.} {\it Let $D$ be an infinite division ring, $n$ a natural number and
$M$ a maximal subgroup of $\mathrm{GL}_n(D)$. If $n=1$ or the center of $F[M]$ contains at least
five elements, then $M$ cannot be finitely generated soluble-by-finite.}\vspace{.3 cm}

Finally, an example of a finitely generated soluble skew linear group is given in~\cite[p.~128]{skew}.
Thus the condition that $M$ is a maximal subgroup of $\mathrm{GL}_n(D)$ in Theorem~B$'$ is essential.\\


\noindent{\bf Acknowledgments.} The research of the second author
was in part supported by a grant from IPM (Grant No. 91050220).
The second author also thanks National Foundation of Elites for
financial support.
{}

\end{document}